\documentclass[]{amsart}

\usepackage{amssymb,amsmath,amsfonts,amsthm,graphicx, amsrefs}

\newtheorem{theorem}{Theorem}[section]
    \newtheorem{lemma}[theorem]{Lemma}
    \theoremstyle{definition}
    
    \theoremstyle{definition}
    
    \theoremstyle{definition}
    \newtheorem{ex}[theorem]{Example}
    \theoremstyle{definition}
    \theoremstyle{remark}
    
    \theoremstyle{remark}
    \newtheorem*{remark*}{Remark}

    \newcommand{\abs}[1]{\left\lvert{#1}\right\rvert}
    \newcommand{\paren}[1]{\left({#1}\right)}

    \newcommand{\ceil}[1]{\left\lceil {#1} \right\rceil}
    \newcommand{\set}[1]{\left\{{#1}\right\}}

\title{$p$-Groups with derived length three and three character degrees}
\author[N. Beike]{Nicolas Beike}
\address{Department of Mathematical Sciences, Kent State University, Kent, OH 44242, USA}
\email{nbeike@kent.edu}
\date{\today}

\begin{document}
    \begin{abstract}
        In this paper we construct examples of $p$-groups with derived length three and three character degrees.
    \end{abstract} 

    \maketitle

    \section{Introduction}

          Throughout this paper all groups are assumed to be finite. Let $G$ be a group. If $G$ is solvable, we define $\mathrm{dl}(G)$ to be the derived length of $G$. For any group $G$ we let $\mathrm{Irr}(G)$ be the set of ordinary irreducible character of $G$ and $\mathrm{cd}(G)=\set{\chi(1) \mid \chi\in \mathrm{Irr}(G)}$ be the set of character degrees of $G$.

          The Isaacs-Seitz Conjecture claims that for any solvable group G, the inequality $\mathrm{dl}(G)\leq {\mathrm{cd}(G)}$ always holds. It is believed that this conjecture is true and that the actual bound will end up being logarithmic instead of linear. The conjecture has been proved when $\abs{\mathrm{cd}(G)}=3,4,5$ by Isaacs \cite{Isa}, Garrison \cite{Gar} in his dissertation, and Lewis \cite{Lew} respectively. 
    
          Here we look at the case where there are 4 character degrees. We are interested in the case where we have equality. So when $\mathrm{dl}(G)=4=\abs{\mathrm{cd}(G)}$. There has been multiple papers regarding this case. We are continuing the work of Du and Lewis \cite{DuLew} and looking at one of the few remaining cases where the character degrees are powers of a single prime. Our goal is to classify all groups of this type. In this case we have that $G=V P$ where $P$ is a $p$-group and $V$ is a normal $p$-complement and the product is either direct or semidirect. In the semidirect case the situation can be reduced to $P$ being a $p$-group with derived length 3 and having 3 character degrees. We are interested in understanding the $p$-groups here. 

          In this paper, we construct a few examples of $p$-groups with derived length 3 and 3 character degrees. Most of the examples are groups where $\abs{P}=p^6$. In section 2 we discuss the examples given in chapter 26 of Huppert's character theory text \cite{Hup}. These examples are constructed as extensions of extraspecial groups of order $p^{2n+1}$ for $n\geq 2$ and $p\geq 5$ of exponent $p$. The extensions are by cyclic groups of order $p$. In Section 3 we discuss an example given by Noritzsch in \cite{Nor}. This is an extension of the extraspecial group of order $3^5$ of exponent $3$ by a cyclic group of order $9$. In Theorem 3.1 we show that for other odd primes an extension by a cyclic group of order $p^2$ is not possible. 
          
          In section 4 we give an outline of a method to constuct more examples. This involves taking a group of order $p^5$ with derived length 2 and extending it by a cyclic group of order $p$. The automorphism constructed has to satisfy a few different conditions to work. Lemma 4.1 gives us a way to compute the order of the automorphisms we construct without having to compute by hand. In section 5 we give two more constructions for groups with derived length 3 and 3 character degrees. Example 5.1 extends the extraspecial groups of order $p^5$ of exponent $p^2$ where $p\geq 5$. Example 5.2 steps away from the previous examples by extending a nonextraspecial group. The presentation for the base group is given later but for $p=7$ this is SmallGroup(16807,72) \cite{Smallgrp}. There are 10 base groups that we know of so far that can be extended in a similar way. 
          
          For the remainder of the paper if $\alpha :G\to G$ for a group $G$, then we let $g^{\alpha}=\alpha(g)$ for all elements $g\in G$. We also let $[g,\alpha]=g^{-1}g^{\alpha}$.

          This work was completed while the author was a Ph.D student at Kent State University under the supervision of Professor Mark Lewis. All or part of this work may appear as part of the author's Ph.D dissertation.

    \section{Huppert's Examples}

        Huppert in \cite{Hup} 26.3(b) gives examples of $p$-groups with derived length 3 and 3 character degrees. 

        In the first example we will construct a family of groups with 3 character degrees and derived length 3 that are extensions of extraspecial groups of exponent $p$. We require that the order of the extraspecial group is at least $p^{7}$ and $p>3$.
        
        \begin{ex}
            Let $G$ be an extraspecial group of order $p^{2n+1}$ and exponent $p$. Let $p\geq 3$ and $n\geq 3$. Write $G=\langle x_1,x_2,...,x_n,y_1,y_2,...,y_n\rangle$ where $[x_i,x_j]=1$, $[y_i,y_j]=1$, $[x_i,y_i]=c$ and $[x_i,y_j]=1$ for $i\neq j$. Define a map $\alpha:G\to G$ by 
            
            \begin{align*}
                \alpha(x_1)&=x_1\\
                \alpha(x_2)&=x_1x_2\\
                \alpha(x_3)&=x_2x_3 \\
                \alpha(x_i)&=x_i \text{ for } i>3\\
                \alpha(y_1)&=y_1y_2^{-1}y_3\\
                \alpha(y_2)&=y_2y_3^{-1} \\
                \alpha(y_i)&=y_i \text{ for } i>2
            \end{align*}

            Now $\alpha$ is surjective and it is not hard to check that $[x_i^{\alpha},y_j^{\alpha}]=[x_i,y_j]^{\alpha}$, $[x_i^\alpha,x_j^{\alpha}]=[x_i,x_j]^{\alpha}$ and $[y_i^{\alpha},y_j^{\alpha}]=[y_i,y_j]^{\alpha}$. Hence $\alpha$ is an automorphism of $G$. 
        
            Also we have that $\alpha^p$ acts trivially on all generators and $\alpha^n$ acts nontrivial for $0<n<p$. Thus $\alpha^p$ is the identity automorphism and has order $p$. 

            Construct $P=G\rtimes \langle \alpha \rangle$. Then $P'=\langle c, x_1, x_2, y_2,y_3\rangle$ and $P''=\langle x \rangle$. Now $P/P''$ is non-abelian and has normal abelian subgroup of index $p$ namely $G/P''=G/G'$. So by It\^{o}'s theorem (see 6.15 of \cite{Isa2}) $\chi(1)\mid p$ for all $\chi \in \mathrm{Irr}(P/P'')$. Thus $\mathrm{cd}(P/P'')=\set{1,p}$. Now let $\chi \in \mathrm{Irr}(G)$ such that $\chi(1)=p^n$. Then since $\chi$ vanishes off $Z(G)$ and $P/G$ is cyclic, then $\chi$ extends to $P$. In fact $\chi$ has $p$ such extensions and there are $p-1$ characters like $\chi$ that extend. Hence there are $p(p-1)$ characters of $P$ with degree $p^n$ with $Z(G)=G'=P''$ not in the kernel. Now observe that 
            
            \begin{align*}
                |P|&=\sum_{\chi \in \mathrm{Irr}(P)} \chi(1)^2\\
                &=p^{2n+2}\\
                &=p^{2n+1}+p^{2n+2}-p^{2n+1}\\
                &=p^{2n+1}+p(p-1)p^{2n}\\
                &=|P/P''|+p(p-1)p^{2n}\\
                &=\sum_{\chi\in \mathrm{Irr}(P/P'')}\chi(1)^2+\sum_{\substack{\chi| \chi\in \mathrm{Irr}(G) \\ \text{extends to } P}} \chi(1)^2
            \end{align*}

            This describes all irreducible characters of $P$ and so $\mathrm{cd}(P)=\set{1,p,p^n}$. Hence $P$ is a group with derived length 3 and has 3 character degrees. 

            Thus we have constructed the desired family of groups.

        \end{ex}
        In the next example we follow a similar construction but work with the extraspecial group of order $p^5$ of exponent $p$ where $p>3$.
        
        \begin{ex}
            Let $G$ be an extraspecial group of order $p^5$ and exponent $p$ where $p>3$. Write $G=\langle a,b,x,y,c\rangle$ where $[a,b]=c$, $[x,y]=c$ and all other commutators are trivial. 

            Define a map $\alpha : G\to G$ by 
            
            \begin{align*} 
                \alpha(a)&=a\\
                \alpha(b)&=by^{-1}\\
                \alpha(x)&=ax\\
                \alpha(y)&=a^{-1}x^{-1}y\\
                \alpha(c)&=c
            \end{align*}

            Then $\alpha$ is surjective and $[g_1^{\alpha},g_2^{\alpha}]=[g_1,g_2]^{\alpha}$ for all generators $g_1,g_2$ of $G$. Hence $\alpha$ is an automorphism of $G$. Also $\alpha^p$ acts trivially and so $\alpha^3=1$. 
            
            Let $P=G\rtimes \langle \alpha \rangle$. Then by a similar argument as above, $P$ has derived length 3 and 3 character degrees. Alternatively we can observe that $P''>1$ thus $P$ has derived length $3$. Since $\abs{P}=p^6$,  $\mathrm{cd} (P)\subseteq \set{1,p,p^2}$. $P$ is an $M$-group so $\mathrm{dl}(P)\leq \abs{\mathrm{cd}(P)}$. Thus $P$ must have 3 character degrees.

            Thus we have constructed the desired example. Putting 2.1 and 2.2 together, we have examples that are extensions of extraspecial groups of order $p^{2n+1}$ for $n\geq 2$, $p\geq 5$, of exponent $p$.

        \end{ex}

    \section{Noritzsch's Example}

        Noritzsch \cite{Nor} gives another example of a $p$-group with derived length 3 and 3 character degrees. We will construct this $p$-group from an extraspecial group of order $3^5$ being acted on by a cyclic group of order $9$.
        \begin{ex}
    
            Start with the extraspecial group of order $3^5$ with exponent $3$, $G$. Let $G=\langle a,b,x,y,c\rangle $ where $[a,b]=c$, $[x,y]=c$ and $a^3=c$. All other commutators are trivial and $b,x,y,c$ have order $3$. Now define a map $\alpha:G\to G$ by 
            
            \begin{align*} 
                \alpha(a)&=y\\
                \alpha(b)&=x^2b^2\\
                \alpha(x)&=aby\\
                \alpha(y)&=a^2y^2\\
                \alpha(c)&=c 
            \end{align*}

            Then similar to before, $\alpha$ is surjective and $[g_1^{\alpha},g_2^{\alpha}]=[g_1,g_2]^{\alpha}$ for all generators $g_1,g_2$ of $G$. Thus $\alpha$ is an automorphism of $G$. Also $\alpha^9$ acts trivially which $\alpha^3$ does not so $\alpha^9=1$. Let $P=G\rtimes \langle \alpha \rangle$. Then $P''=G'$ and so $G/P''$ is a normal abelian subgroup of $P/P''$. Thus $\chi(1)\mid \abs{P:G}=3^2$ for all $\chi\in \mathrm{cd}(P/P'')$. Hence $\mathrm{cd}(P/P'')\subseteq \set{1,3,3^2}$. Like before we have that the nonlinear irreducible characters of $G$ extend to $P$. There are exactly $3^2(3-1)$ irreducible characters of $P$ which are these extensions, each having degree $3^2$. Now 

            \begin{align*}
                \abs{P}&= 3^7\\
                &= 3^6+3^2(3-1)3^4\\
                &=\abs{P/P''}+\sum_{\substack{\chi| \chi\in \mathrm{Irr}(G) \\ \text{extends to } P}} \chi(1)^2
            \end{align*}

             This describes all irreducible characters of $P$ which are the irreducible characters of $P/P''$ and the characters of $G$ that extend to $P$. Thus $\mathrm{cd}(P)=\set{1,3,9}$. So $P$ is a group with derived length 3 and has 3 character degrees, completing the example.
        \end{ex}

        We note here that there do exist extensions of the extraspecial group of order $3^5$ and exponent $9$ like in the previous example. So the exponent of the group does not play a part. However the example does depend on the prime $3$ as shown by Theorem 3.3. First we need a lemma.

        \begin{lemma}
            If $G$ is a $p$-group where $G'$ is abelian, then 

            \[
                (gh)^n=g^nh^n[g,h]^{\binom{n}{2}} [g,h,h]^{\binom{n}{3}}...[g,h,...,h]^{\binom{n}{n}}
            \]
            where the last term has $n-1$ copies of $h$. 
        \end{lemma}

        \begin{proof}
            This follows by an easy induction argument inducting on $n$.
        \end{proof}

        \begin{theorem}
            Let $G$ be an extraspecial group of order $p^{2n+1}$. Then for $p> 2n$ there does not exist an automorphism of $G$ of order $p^n$ for $n>1$. In particular for $n=5$, $p=3$ is the only group with an automorphism of order $p^2$ for odd primes. 
        \end{theorem}

        \begin{proof}
            If we let $G$ be an extraspecial group of order $p^{2n+1}$. Then let $H \leq \mathrm{Aut}(G)$ be the automorphisms that fix $\mathrm{Z}(G)$. Then by Theorem 1 of \cite{Win} $H/\mathrm{Inn}(G)$ is a subgroup of $\mathrm{Sp}(2n,p)$ which is a subgroup of $\mathrm{GL} (2n,p)$. The Sylow $p$-subgroups of $\mathrm{GL} (n,p)$ has exponent $p^{\ceil{\log_p(n)}}$. Thus for $p\geq 2n$ we have that the Sylow $p$-subgroups of $\mathrm{GL} (n,p)$ have exponent $p$. Thus so does $H/\mathrm{Inn}(G)$.

            It is known that for $p\geq 3$ that $H=I\rtimes S$. This is a folklore result and is easily proven. See this Math StackExchange post by Derek Holt with an outline of the proof \cite{Holt}. Let $\alpha\in \mathrm{Aut}(G)$ with $o(\alpha)=p^n$. Then $\alpha$ acts trivially on $Z(G)$. Hence $\alpha \in H$. So $\alpha=\sigma \phi$ where $\sigma \in I$ and $\phi \in S$. Let $T=I\rtimes \langle \phi\rangle$. Then $T'=I'[I,\phi]\langle \phi\rangle'=[I,\phi]\leq I$ Since $I$ is abelian, so is $T'$. Additionally $\abs{T}=p^{2n+1}$ so $\mathrm{cl}(T)\leq 2n<p$ since $p$ is an odd prime. Now $[T,\alpha]\leq T'\leq I$ has exponent $p$. Hence by the previous lemma we have the following 
            \[
                \alpha^p=\sigma^p \phi^p [\sigma, \phi]^{\binom{p}{2}}... [\sigma, \phi, ..., \phi]^{\binom{p}{p}}
            \]
            where the last term has $p-1$ copies of $\phi$. Since $\mathrm{cl}(T)<p$ and $\mathrm{exp}([T,\phi])=p$ we are left with 
            \[
                \alpha^p=\sigma^p\phi^p=1
            \]
            Since $I$ has exponent $p$ and so does $S\cong H/I$. Hence $\alpha$ has order $p$ completing the proof.
        \end{proof}

        Thus this shows that Noritzsch's example does rely on the prime being $3$.

    \section{Outline for Constructions}

        In this section we give a brief outline for the construction of more groups with derived length 3 and 3 character degrees. In particular we take $p$-groups of order $p^5$ and extend them by a cyclic group of order $p$ for the constructions. 
        
        Let $G$ be a $p$-group of order $p^5$ that has derived length 2. Construct an automorphism $\alpha$ of $G$ with order $p$ such that in the semidirect product, $P=G\rtimes \langle \alpha \rangle$, we have $P''>1$. Then $P$ has derived length 3 and at least 3 character degrees. Since $P$ has order $p^6$, it cannot have more than 3 character degrees. In general we do the following steps to find an automorphism with the required properties. Each step is dealt with differently for different groups:

        \begin{enumerate}
            \item Let $G$ be a $p$-group of order $p^5$ and derived length $2$. 
            \item Write $G=\langle a,b,x,y,c | Relations \rangle$. 
            \item Define $\alpha:G\to G$ on the generators of $G$.
            \item Check that $\alpha$ satisfies the relations of $G$ i.e. $[g,h]^{\alpha}=[g^{\alpha},h^{\alpha}]$. Then $\alpha$ is a homomorphism.
            \item Find sufficient conditions so that $\alpha$ is surjective. Since $G$ is finite, then $\alpha$ is now an automorphism.
            \item Check that $P''>1$ where $P=G\rtimes \langle \alpha \rangle$.
            \item Check that $\alpha$ has order $p$. 
        \end{enumerate}

        Here is a useful lemma we use to compute the order of the automorphism in step 7 in our examples and avoid a long and tedius calculation.

        \begin{lemma}
            Let $\alpha$ be an automorphism of $G$ and suppose $G$ has nilpotence class at most 5. Then for $g\in G$ we have
            \[
                g^{\alpha^n}=g[g,\alpha]^{\binom{n}{1}}[g,\alpha,\alpha]^{\binom{n}{2}} [g,\alpha,\alpha,\alpha]^{\binom{n}{3}} [g,\alpha,\alpha,\alpha,\alpha]^{\binom{n}{4}} [[g,\alpha,\alpha],[g,\alpha]]^{\binom{n}{3}}
            \]  
            In particular if $[G,\alpha]$ has exponent $p$ and $p\geq 5$, then $o(\alpha)$ divides $p$.
        \end{lemma}

        \begin{proof}

            We will make use of the identity $g^{\alpha}=\alpha^{-1}g\alpha=g[g,\alpha]$ which holds for any element of $g$ when looking from the viewpoint of $G\rtimes \langle \alpha \rangle$.

            Let $[g,\alpha]_n=[g,\alpha,...,\alpha]$ where on the right hand side there are $n$ copies of $\alpha$. Then we have $([g,\alpha]_n)^{\alpha}=[g,\alpha]_n[g,\alpha]_{n+1}$. Since the group has class at most $5$ and $[[g,\alpha]_n,[g,\alpha]_{n+1}]\in [G_{n+1},G_{n+2}]\leq G_{2n+3}$, we have $([g,\alpha]_n^k)^{\alpha}=[g,\alpha]_n^k[g,\alpha]_{n+1}^k$ for $n>1$. For $n=1$ we have the following
            \begin{align*}
                ([g,\alpha]^k)^\alpha&=\paren{[g,\alpha][g,\alpha,\alpha]}^k\\
                &=[g,\alpha]^k[g,\alpha,\alpha]^k[[g,\alpha][g,\alpha,\alpha]]^{\binom{k-1}{2}}
            \end{align*}
            Where the second equality holds since $[[g,\alpha][g,\alpha,\alpha]]\in G_5$ and $[G_5,G]=G_6=1$

            Now the result follows by induction on $n$ and noting that the exponent on $[[g,\alpha][g,\alpha,\alpha]]$ is $\binom{n-1}{2}+\binom{n-1}{3}=\binom{n}{3}$ and $[[g,\alpha][g,\alpha,\alpha]]^{\alpha}=[[g,\alpha][g,\alpha,\alpha]]$.

            Now if $[G,\alpha]$ has exponent $p$ and $p\geq 5$ then taking $n=p$ we have that all terms in the identity  except for the first term are equal to 1. So $\alpha^p$ is the identity map. Thus $o(\alpha)\mid p$. 
        \end{proof}

    \section{Explicit Constructions}
        In this section we go through the construction of a few families of the groups in question. It is known that for primes $p=2,3$ all groups of order $p^6$ have derived length 2. Hence we consider primes $p\geq 5$.

        For our first example we construct a family of groups with 3 character degrees and derived length 3 that are extensions of extraspecial groups of exponent $p^2$

        \begin{ex}
            \textbf{(Step 1 and 2)} Let $G$ be an extraspecial group of order $p^5$ with exponent $p^2$. Write $G=\langle a,b,x,y,c \rangle$ where $[a,b]=c$, $[x,y]=c$, $a^p=c$, all other generators have order $p$ and all other commutators are trivial. 

            \textbf{(Step 3)} Define a map $\alpha:G\to G$ by 

            \begin{align*}
                &[a,\alpha]=b^{n_1}x^{n_2}y^{n_3} &&\text{or } &a^\alpha&=ab^{n_1}x^{n_2}y^{n_3}\\
                &[b,\alpha]=1 &&\text{or } &b^\alpha&=b\\
                &[x,\alpha]=b^{n_4}x^{n_5}y^{n_6} &&\text{or } &x^\alpha&=xb^{n_4}x^{n_5}y^{n_6}\\
                &[y,\alpha]=b^{n_7}x^{n_8}y^{n_9} &&\text{or } &y^\alpha&=yb^{n_7}x^{n_8}y^{n_9}\\
                &[c,\alpha]=1 &&\text{or } &c^\alpha&=c
            \end{align*}
            Where $n_1,...,n_9$ are integers (mod $p$). Right now there are no restrictions on $n_1,... n_9$. As we go through the remaining steps we will acquire some restrictions.

            \textbf{(Step 4)} Now we check the commutator relations. Since $c$ is central we have $[g^{\alpha},c^{\alpha}]=[g^{\alpha},c]=1=1^\alpha=[g,c]^\alpha$. So we are left to check the commutators not involving $c$. In the following we expand the commutators and collect variables. 

            \begin{align*}
                \begin{split}
                    [a^\alpha,b^\alpha]&=c
                \end{split}
                \intertext{We want this to be equal to $c$. So we get no restriction here.}
                \begin{split}
                    [a^\alpha,x^\alpha]&=c^{n_2 n_6+n_4-n_3-n_3 n_5}
                \end{split}
                \intertext{We want this to be 1. So we need $n_2 n_6+n_4-n_3-n_3 n_5\equiv 0 \pmod{p}$.}
                \begin{split}
                    [a^\alpha,y^\alpha]&=c^{n_2n_9+n_7+n_2-n_3n_8}
                \end{split}
                \intertext{We want this to be equal to 1. So we need $n_2n_9+n_7+n_2-n_3n_8\equiv 0 \pmod{p}$.}
                \intertext{Since $b$ commutes with $x$ and $y$ it is easy to see $[b^{\alpha},x^{\alpha}]=[b,x]^{\alpha}$ and $[b^{\alpha},y^{\alpha}]=[b,y]^{\alpha}$. Thus giving no additional restrictions}
                \begin{split}
                    [x^\alpha,y^\alpha]&=c^{1+n_5n_9+n_5+n_9-n_6n_8}
                \end{split}
                \intertext{We want this to equal $c$. So we need $n_5n_9+n_5+n_9-n_6n_8\equiv 0 \pmod{p}$.}
            \end{align*}

            Provided $n_1,...,n_9$ satisfy these conditions, then $\alpha$ is a homomorphism.

            \textbf{(Step 5)}

            Now we find sufficient conditions on $n_1,...,n_9$ so that $\alpha$ is surjective and thus an automorphism. If we let $n_5,n_8,n_9=0$ then clearly $\alpha$ is surjective.

            To summarize, the map $\alpha$ is now

            \begin{align*}
                &[a,\alpha]=b^{n_1}x^{n_2}y^{n_3} &&\text{or } &a^\alpha&=ab^{n_1}x^{n_2}y^{n_3}\\
                &[b,\alpha]=1 &&\text{or } &b^\alpha&=b\\
                &[x,\alpha]=b^{n_4}y^{n_6} &&\text{or } &x^\alpha&=xb^{n_4}y^{n_6}\\
                &[y,\alpha]=b^{n_7} &&\text{or } &y^\alpha&=yb^{n_7}\\
                &[c,\alpha]=1 &&\text{or } &c^\alpha&=c
            \end{align*}

            Where $n_2n_6+n_4-n_3\equiv 0 \pmod{p}$, and $n_7+n_2\equiv 0 \pmod{p}$.

            \textbf{(Step 6)}

            Let $P=G\rtimes \langle \alpha \rangle$. We need to ensure that $P''>1$. Computing the derived subgroup of $P$, we have $P'=\langle c, b^{n_1}x^{n_2}y^{n_3}, b^{n_4}y^{n_6}, b^{n_7} \rangle$. If $n_2$ or $n_6$ are $0$ then $P''=1$. Thus we need both $n_2,n_6\neq 0$ and then $P'=\langle b,x,y,c \rangle$ and $P''=\langle c \rangle$. 

            \textbf{(Step 7)}

            First we will compute the lower central series for $P$. 

            \begin{align*}
                P&=\langle a,b,x,y,c,\alpha \rangle\\
                P_2&=\langle b,x,y,c \rangle\\
                P_3&=\langle b,y,c\rangle \\
                P_4&=\langle b,c\rangle\\
                P_5&=\langle c \rangle\\
                P_6&=\langle 1 \rangle
            \end{align*}

            Hence $P$ has nilpotence class $5$ and $[P,\alpha]$ has exponent $p$. By Lemma 4.1 this implies that $\alpha$ has order $p$ as desired. Therefore $P$ is a group of order $p^6$ with derived length 3 and 3 distinct character degrees, giving the desired example.
        \end{ex}

        For the second example we will deviate from the other examples and construct an extension of a group that is not extraspecial with the desired properties. We take a $p$-group, whose presentation is given below, of order $p^5$ and extend by a cyclic group of order $p$. For $p=7$ this initial $p$-group is SmallGroup(16807,72) in the small groups library.

        \begin{ex}
            \textbf{(Steps 1 and 2)} 
            In 

            Let $G=\langle a,b,x,y,c\rangle$ where $[a,b]=c$, $[a,y]=b$, $[x,y]=c$, all other commutators are trivial, and $G$ has exponent $p$. 

            \textbf{(Step 3)}
            
            Define a map $\alpha:G\to G$ by 

            \begin{align*}
                &[a,\alpha]=b^{n_1}x^{n_2}y^{n_3} &&\text{or } &a^\alpha&=ab^{n_1}x^{n_2}y^{n_3}\\
                &[b,\alpha]=1 &&\text{or } &b^\alpha&=b\\
                &[x,\alpha]=b^{n_4}x^{n_5}y^{n_6} &&\text{or } &x^\alpha&=xb^{n_4}x^{n_5}y^{n_6}\\
                &[y,\alpha]=b^{n_7}x^{n_8}y^{n_9} &&\text{or } &y^\alpha&=yb^{n_7}x^{n_8}y^{n_9}\\
                &[c,\alpha]=1 &&\text{or } &c^\alpha&=c
            \end{align*}

            Similar to before $n_1,...,n_9$ are integers (mod p) and have no restrictions. As we proceed we will acquire some restrictions.

            \textbf{(Step 4)}

            Now we check the commutator relations. Like before, since $c$ is central we do not need to check commutators involving $c$.
            \begin{align*}
                \begin{split}
                    [a^\alpha,b^\alpha]&=c
                \end{split}
                \intertext{We want this to be equal to $c$ so we get no restrictions here.}
                \begin{split}
                    [a^\alpha,x^\alpha]&=c^{n_4-n_3-n_3n_5}
                \end{split} 
                \intertext{We want this to equal $1$ so we get the restriction that $n_4-n_3-n_3n_5\equiv 0 \pmod{p}$.}
                \begin{split}
                    [a^\alpha,y^\alpha]&=bc^{n_2+n_7-n_3n_8}
                \end{split}
                \intertext{We want this to equal $b$ so we need $n_2+n_7-n_3n_8\equiv 0 \pmod{p}$.}
                \intertext{Since $b$ commutes with $x$ and $y$ it is easy to see that $[b^\alpha,x^\alpha]=1=[b,x]^\alpha$ and $[b^\alpha,y^\alpha]=1=[b,y]^\alpha$.}
                \begin{split}
                    [x^\alpha,y^\alpha]&=c^{1+n_5}
                \end{split}
                \intertext{We want this to equal $c$ so we need $n_5=0$.}
            \end{align*}
            
            To summarize, the map $\alpha$ is now 
            \begin{align*}
                &[a,\alpha]=b^{n_1}x^{n_2}y^{n_3} &&\text{or } &a^\alpha&=ab^{n_1}x^{n_2}y^{n_3}\\
                &[b,\alpha]=1 &&\text{or } &b^\alpha&=b\\
                &[x,\alpha]=b^{n_4} &&\text{or } &x^\alpha&=xb^{n_4}\\
                &[y,\alpha]=b^{n_7}x^{n_8} &&\text{or } &y^\alpha&=yb^{n_7}x^{n_8}\\
                &[c,\alpha]=1 &&\text{or } &c^\alpha&=c
            \end{align*}
            Where $n_4-n_3\equiv 0 \pmod{p}$ i.e. $n_4=n_3$ and $n_7-n_3n_8\equiv 0 \pmod{p}$.

            \textbf{(Step 5)} It is easy to see that the map is now surjective. So there are no more restrictions from this step.

            \textbf{(Step 6)} Let $P=G\rtimes \langle \alpha \rangle$. Now $P'=\langle b,c, b^{n_1}x^{n_2}y^{n_3}, b^{n_4}, b^{n_7}x^{n_8} \rangle$. We need $n_3$ and $n_8$ to be nonzero in order for $P''>1$. 

            \textbf{(Step 7)} We can compute the lower central series for $P$ getting

            \begin{align*}
                P&=\langle a,b,x,y,c,\alpha \rangle\\
                P_2&=\langle b,c,x,y \rangle\\
                P_3&=\langle b,c,x \rangle\\
                P_4&=\langle b,c \rangle\\
                P_5&=\langle c \rangle\\
                P_6&=\langle 1 \rangle
            \end{align*}

            Hence $P$ has nilpotence class 5. So by a similar argument to the last example, we have $o(\alpha)=p$. Thus $P$ is a $p$-group with derived length 3 and 3 character degrees. This completes the example.
        \end{ex}

        There are 10 $p$-groups of order $p^5$, including the given examples, that we know of so far that can be extended in a similar manner. These groups come from computations with groups of order $7^5$ in  GAP \cite{GAP4}. The prime 7 was chosen so that it was small enough to do computations relatively fast while not running into issues like with 2 and 3. 

        The constructions given may not contain all isomorphism classes of extensions of these groups that have derived length 3 and 3 character degrees. Also we have seen overlap between the isomorphism classes of the extensions for different initial groups. So two different initial groups may be extended to the same group. We still do not know if there are extensions of extraspecial groups of order $5^7$ by a cyclic group of order 25 that have derived length 3 and 3 character degrees like in Noritzsch's example. The main idea of this paper was to give a way to construct more examples of these groups which by themselves can be hard to compute in GAP. We hope that having explicit constructions of these groups with help further the study of these $p$-groups. These groups are fundamental to the motivating problem in the Introduction, so any information that can be learned is beneficial.

\bibliographystyle{amsplain}
\bibliography{PGroup3cdAnddl3.bib}
\end{document}